\documentclass[12pt]{amsart}

\usepackage{amsmath}
\usepackage{amsfonts}
\usepackage{amssymb,enumerate}
\usepackage{amsthm}
\usepackage{fancyhdr}
\usepackage{tikz}
\usetikzlibrary{arrows,matrix}
\usepackage{hyperref}

\newtheorem{lemma}{Lemma}
\newtheorem{corollary}[lemma]{Corollary}
\newtheorem{proposition}[lemma]{Proposition}
\newtheorem{theorem}[lemma]{Theorem}

\theoremstyle{definition}

\newtheorem{algorithm}[lemma]{Algorithm}

\numberwithin{lemma}{section}

\title[ ]{On the least positive solution to a proportionally modular Diophantine inequality}

\author{A. Moscariello}



\address[A. Moscariello]{Dipartimento di Matematica e Informatica, \ Universit\`a di Catania, \  Viale Andrea Doria 6, 
95125 Catania,Italy}

\email{alessio.moscariello@studium.unict.it}

\begin{document}

\begin{abstract}
Given three positive integers $a,b,c$, a proportionally modular Diophantine inequality is an expression of the form $ax \mod{b} \le cx$. Our aim is to give a recursive formula for the least solution to such an inequality. We  then use the formula to derive an algorithm. Finally, we apply our results to a question of Rosales and Garc\'ia-S\'anchez.

\end{abstract}

\maketitle

\section*{Introduction}
A \emph{proportionally modular Diophantine inequality} is an expression of the form $$ (ax \mod b) \le cx$$ where the positive integers $a,b,c$ are called respectively the \emph{factor}, \emph{modulus} and \emph{proportion}. It is well-known that the set of the non-negative integer solutions of this inequality is a \emph{numerical semigroup} (cf. \cite{RG}, \cite{RGU}), i.e. a submonoid $S$ of $(\mathbb{N},+)$ with finite complement in it.
Denoting by $S(a,b,c)$ the set of solutions, the structure of this set (called a \emph{proportionally modular semigroup}) has been widely studied, but is not completely understood yet. In particular, it is an open problem (cf. \cite{RG}) to find explicit formulas for several classical invariants of numerical semigroups.

In this paper we study the \emph{multiplicity} of a proportionally modular semigroup $S$, that is the smallest positive integer that belongs to $S$. Although some partial results are known (cf. \cite{RGU}, \cite{RU}, \cite{RV}) as of today the main problem of finding a formula for this invariant still remains unsolved. Using elementary number theory we will work on a recursive formula for the smallest positive solution of the more general inequality $$(ax \mod b) \le cx \ \  a,b \in \mathbb{Z}^+, \ c \in \mathbb{Q}^+$$.

Our work is structured as follows: in the first section prove our main theorem, which provides a recursive formula for the computation of the multiplicity of $S$. In Section 2 we describe the algorithm that can be derived from our main theorem. In the final section we explain how our result can be applied to a question of Rosales and Garc\'ia-S\'anchez (\cite[Open Problem 5.20]{RG}).

\section{Main Result}
Given two integers $m$ and $n$ with $n > 0$ we define the \emph{remainder operator} $[m]_n$ as follows   $$[m]_n=\min\{i \in \mathbb{N} \ \ |  \ i \equiv m \pmod n\}$$ Note that if $m$ and $n$ are positive integers such that $m < n$ then $m= [m]_n$. The following properties follow from the definition of floor and ceiling function, and we will use them extensively:
\begin{proposition}\label{mod}
Let $a,b \in \mathbb{Z}^+$. Then:
\begin{enumerate}
\item $\left \lfloor \dfrac{b}{a} \right \rfloor a + [b]_a =b$\\
\item $\left \lceil \dfrac{b}{a} \right \rceil a - [-b]_a =b$
\end{enumerate} 
\end{proposition}

Let $a,b \in \mathbb{Z}^+$ and let $c \in \mathbb{Q}^+$. Consider the inequality $(ax \mod b)=[ax]_b \le cx$, and define $$L(a,b,c)=\min \{x \in \mathbb{Z}^+ \ | \ [ax]_b \le cx \}=\min \{S(a,b,c) \setminus \{0\}\}$$ First of all, we note that if $a \ge b$ then $[ax]_b=[[a]_bx]_b$ and it follows that $L(a,b,c)=L([a]_b,b,c)$, so the condition $a < b$ that we will impose in the next results is not restrictive. Moreover, we have the following property:
\begin{lemma}\label{gcd}
Let $a,b \in \mathbb{Z}^+$, $c \in \mathbb{Q}^+$. Let $d \in \mathbb{Z}^+$ be such that $d \ | \ a$ and $d \ | \ b$. Then $L(a,b,c)=L\left(\frac{a}{d},\frac{b}{d},\frac{c}{d}\right)$.
\end{lemma}
\begin{proof}
Denote $a=da'$ and $b=db'$. In our notation we have $[a]_b=d[a']_{b'}$, therefore $$[ax]_b \le cx \Longleftrightarrow [a']_{b'} \le \frac{c}{d}x$$ and the thesis follows.
\end{proof}
The following proposition gives us the value of $L(a,b,c)$ for two special cases:
\begin{proposition}

Let $a,b \in \mathbb{Z}^+$ be such that $a<b$, and let $c \in \mathbb{Q}^+$ be a positive rational number. Then:
\begin{enumerate}\label{base1}
\item If $c \ge a$ then $L(a,b,c)=1$.
\item If $c < a$ and $a \ | \ b$ then $L(a,b,c)=\frac{b}{a}$.
\end{enumerate}				
\end{proposition}
\begin{proof}
The first one is obvious. As for the second one, if $x<\frac{b}{a}$ then $ax < b$ and $[ax]_b=ax > cx$, so the inequality is false for $x < \frac{b}{a}$, while for $x=\frac{b}{a}$ we have $ax=b$ and $[ax]_b=0 \le cx$. We conclude that $L(a,b,c)=\frac{b}{a}$.
\end{proof}
With these premises we can reduce our problem to the case $c<a<b$, $a  \not | \ b$.
\begin{proposition}\label{l}
Let $a,b \in \mathbb{Z}^+$ and $c \in \mathbb{Q}^+$ be such that $c<a<b$ and $a \not | \ b$. Then there exists $\mu \in \mathbb{Z}^+$ such that $$L(a,b,c)=\left \lceil \frac{\mu b}{a} \right \rceil.$$
\end{proposition}

\begin{proof}
If $x < \left \lceil \frac{b}{a} \right \rceil$ then $ax < b$ and $[ax]_b=ax>cx$, so $L(a,b,c) \ge \left \lceil \frac{b}{a} \right \rceil$. From this bound it follows that there exists $\mu \in \mathbb{Z}^+$ such that $$\left \lceil \frac{\mu b}{a} \right \rceil \le L(a,b,c) < \left \lceil \frac{(\mu+1) b}{a} \right \rceil$$ Suppose now that $L(a,b,c) \neq  \left \lceil \frac{\mu b}{a} \right \rceil$, that is equivalent to saying that there exists $r \in \mathbb{N}$, $r \neq 0$ such that $$L(a,b,c) = \left \lceil \frac{\mu b}{a} \right \rceil+r \ \ \ \ \ \text{     with } r < \left \lceil \frac{(\mu+1) b}{a} \right \rceil - \left \lceil \dfrac{\mu b}{a} \right \rceil$$ Therefore $$aL(a,b,c) = a\left \lceil \frac{\mu b}{a} \right \rceil + ar < a\left \lceil \frac{(\mu+1) b}{a} \right \rceil \Longrightarrow [aL(a,b,c)]_b \ge a$$ and since $b > [aL(a,b,c)]_b-a \ge 0$ we obtain $[aL(a,b,c)]_b-a=[aL(a,b,c)-a]_b$.  After the position $x=L(a,b,c)-1$ we have that $$[ax]_b=[a(L(a,b,c)-1)]_b=[aL(a,b,c)-a]_b=[aL(a,b,c)]_b-a$$ and $cx=cL(a,b,c)-c$. Hence we have $$[ax]_b=[aL(a,b,c)]_b-a < [aL(a,b,c)]_b-c \le cL(a,b,c)-c=cx$$ leading to $x=L(a,b,c)-1 \in S(a,b,c)$, which is a contradiction.
\end{proof}
Note that by definition it's clear that $L(a,b,c) \le b$, hence $1 \le \mu \le a$.  Define $R_\mu$ as the only positive integer such that $$\frac{(R_\mu-1) a}{[b]_a} < \mu \le  \frac{R_\mu a}{[b]_a}.$$
\begin{lemma}\label{valori}
Let $a,b \in \mathbb{Z}^+$ and $c \in \mathbb{Q}^+$ be such that $c < a < b$ and $a \not | \ b$. Let $\mu \in \mathbb{Z}^+$ be such that $L(a,b,c)=\left \lceil \frac{\mu b}{a} \right \rceil$. Then we have: 
\begin{enumerate}
\item $L(a,b,c)=\left \lceil \frac{\mu b}{a} \right \rceil=\mu\left \lfloor \frac{b}{a} \right \rfloor+R_\mu$
\item $[aL(a,b,c)]_b=R_\mu a -\mu[b]_a$
\end{enumerate}
\end{lemma}
\begin{proof}
\begin{enumerate}
\item By using Proposition \ref{mod} we have that $b=\left \lfloor \frac{b}{a} \right \rfloor a + [b]_a$, and then $$L(a,b,c)=\left \lceil \frac{\mu b}{a} \right \rceil=\left \lceil \frac{\mu \left(\left \lfloor \frac{b}{a} \right \rfloor a +[b]_a \right)}{a} \right \rceil=\left \lceil \mu \left \lfloor \frac{b}{a} \right \rfloor  +\frac{\mu [b]_a}{a} \right \rceil$$
Since $\mu \left \lfloor \frac{b}{a} \right \rfloor \in \mathbb{Z}^+$ we can deduce easily  from the definition of $R_\mu$ that $R_\mu=\left \lceil \frac{\mu [b]_a}{a} \right \rceil$. Then it follows that: $$L(a,b,c)= \left \lceil \mu \left \lfloor \frac{b}{a} \right \rfloor  +\frac{\mu [b]_a}{a} \right \rceil=\mu \left \lfloor \frac{b}{a} \right \rfloor  +R_\mu.$$
\item By the first part of this lemma and the identity $b-[b]_a=\left \lfloor \frac{b}{a} \right \rfloor a$ of Proposition \ref{mod} we have $$[aL(a,b,c)]_b=\left[ \mu\left \lfloor \frac{b}{a} \right \rfloor a+R_\mu a \right]_b=\left[ \mu b - \mu [b]_a+R_\mu a \right]_b=\left[ R_\mu a - \mu [b]_a \right]_b.$$ 
But by definition of $R_\mu$ we have $0 \le R_\mu-\mu[b]_a \le [b]_a < b$ and consequently $\left[ R_\mu a - \mu [b]_a \right]_b=R_\mu a - \mu [b]_a$, that is our thesis.
\end{enumerate}
\end{proof}

In order to find a recursion we'll prove that $R_\mu$ itself is the smallest solution of another proportionally modular Diophantine inequality with smaller values of factor, modulus and proportion, and then we'll compute $\mu$ from $R_{\mu}$:
\begin{theorem}\label{main}
Let $a,b \in \mathbb{Z}^+$, $c \in \mathbb{Q}^+$ be such that $c<a<b$ and $a \not | \ b$. Let $\mu \in \mathbb{Z}^+$ be such that $L(a,b,c)=\left \lceil \frac{\mu b}{a} \right \rceil$. Then

$$R_\mu=L\Bigg([a]_{[b]_a},[b]_a,\frac{cb}{c\left \lfloor \frac{b}{a} \right \rfloor+[b]_a}\Bigg) \ \ \ \
\mu=\left \lceil \frac{R_\mu (a-c)}{c\left \lfloor \frac{b}{a} \right \rfloor+[b]_a} \right \rceil$$
\end{theorem}
\begin{proof}
Using Lemma \ref{valori} we have that $cL(a,b,c)=c\mu \left \lfloor \frac{b}{a} \right \rfloor+R_\mu c$ and $[aL(a,b,c)]_b=R_\mu a - \mu[b]_a$. Then, from $cL(a,b,c) \ge [aL(a,b,c)]_b$ it follows
$$cL(a,b,c) \ge [aL(a,b,c)]_b \Longrightarrow c\mu\left \lfloor \frac{b}{a} \right \rfloor+R_\mu c \ge R_\mu a - \mu[b]_a \Longrightarrow$$ $$\mu\left(c\left \lfloor \frac{b}{a} \right \rfloor+[b]_a \right)\ge R_\mu (a-c) \Longrightarrow \mu \ge \frac{ R_\mu (a-c)}{c\left \lfloor \frac{b}{a} \right \rfloor+[b]_a }.$$ By definition of $R_{\mu}$ we get $\mu \le  \frac{R_\mu a}{[b]_a}$ and by merging these parts we obtain 
\begin{equation}\label{small}
\frac{ R_\mu (a-c)}{c\left \lfloor \frac{b}{a} \right \rfloor+[b]_a} \le \mu \le  \frac{R_\mu a}{[b]_a}.
\end{equation} 
Since $L(a,b,c)$ is the smallest integer for which the inequality is verified, then it follows that $R_{\mu}$ must be the smallest integer such that the interval defined by the two sides of inequuality (\ref{small}) contains an integer . Furthermore, by definition $\mu$ is the smallest integer in such an interval, so we obtain 

\begin{equation}\label{Rmu}
R_\mu=\min \left\{z \in \mathbb{Z}^+ \ | \ \left[ \frac{z(a-c)}{c\left \lfloor \frac{b}{a} \right \rfloor+[b]_a} , \frac{za}{[b]_a} \right] \cap \mathbb{N} \neq \emptyset  \right\}.
\end{equation} and 
\begin{equation}\label{mu}
\mu=\min\left\{ \left[ \frac{R_{\mu}(a-c)}{c\left \lfloor \frac{b}{a} \right \rfloor+[b]_a} ,\frac{R_{\mu}a}{[b]_a} \right] \cap \mathbb{N} \right \}=\left \lceil \frac{R_\mu (a-c)}{c\left \lfloor \frac{b}{a} \right \rfloor+[b]_a} \right \rceil .
\end{equation} 
The second part of our thesis is proved in (\ref{mu}).
For the first one, we can easily see that $$\left[ \frac{z(a-c)}{c\left \lfloor \frac{b}{a} \right \rfloor+[b]_a} , \frac{za}{[b]_a} \right] \cap \mathbb{N} \neq \emptyset \Longleftrightarrow \left \lfloor \frac{za}{[b]_a}\right \rfloor \ge \frac{z(a-c)}{c\left \lfloor \frac{b}{a} \right \rfloor+[b]_a}.$$
By Proposition \ref{mod} we get the two identities  $\left \lfloor \frac{za}{[b]_a}\right \rfloor=\frac{za-[za]_{[b]_a}}{[b]_a}$ and $\left \lfloor \frac{b}{a}\right \rfloor=\frac{b-[b]_{a}}{a}$. Therefore $$\left \lfloor \frac{za}{[b]_a}\right \rfloor \ge z\frac{a-c}{c\left \lfloor \frac{b}{a} \right \rfloor+[b]_a} \Longleftrightarrow \frac{za-[za]_{[b]_a}}{[b]_a} \ge z\frac{a-c}{c\left \lfloor \frac{b}{a} \right \rfloor+[b]_a} \Longleftrightarrow $$ $$ \Longleftrightarrow z\left(\frac{a}{[b]_a}-\frac{a-c}{c\left \lfloor \frac{b}{a} \right \rfloor+[b]_a}\right) = z\left(\frac{ac\left \lfloor \frac{b}{a} \right \rfloor+c[b]_a}{[b]_a\left(c\left \lfloor \frac{b}{a} \right \rfloor+[b]_a\right)}\right) \ge \frac{[za]_{[b]_a}}{[b]_a} \Longleftrightarrow z\left(\frac{cb}{c\left \lfloor \frac{b}{a} \right \rfloor+[b]_a}\right) \ge [za]_{[b]_a}.$$
Now we put this condition in Eq. (\ref{Rmu}): $$R_\mu= \min \left\{z \in \mathbb{Z}^+ \ | \ z\left(\frac{cb}{c\left \lfloor \frac{b}{a} \right \rfloor+[b]_a}\right) \ge [za]_{[b]_a}  \right\}=L \left([a]_{[b]_a},[b]_a,\frac{cb}{c\left \lfloor \frac{b}{a} \right \rfloor+[b]_a}\right)$$ and the thesis is thus proved. 
\end{proof}
Combining Proposition \ref{l} and Theorem \ref{main} we obtain a recursive formula for $L(a,b,c)$:
\begin{corollary}\label{corollary}
Let $a,b \in \mathbb{Z}^+$, $c \in \mathbb{Q}^+$ be such that $c<a<b$ and $a \not | \ b$. Then $$L(a,b,c)= \left \lceil \left \lceil  \frac{L_1 (a-c)}{c\left \lfloor \frac{b}{a} \right \rfloor+[b]_a} \right \rceil \frac{b}{a} \right \rceil \hspace{0.75cm} \text{where} \hspace{0.25cm}L_1=L \left([a]_{[b]_a},[b]_a,\frac{cb}{c\left \lfloor \frac{b}{a} \right \rfloor+[b]_a}\right)$$
\end{corollary}
\section{The Algorithm}
The main result of the previous section gives rise to the following algorithm to compute $L(a,b,c)$ for any given triple $(a,b,c)$ with $a,b \in \mathbb{Z}^+$ and $c \in \mathbb{Q}^+$.
\begin{algorithm}\label{algoritmo}{\ }
\begin{description}
\item [\textbf{Input }] The values $a,b \in \mathbb{Z}^+$, $c \in \mathbb{Q}$ (condition: $a<b$).
\item [\textbf{Output}] The value $L(a,b,c)=\min \{x \in \mathbb{Z}^+ \ | \ [ax]_b \le cx \}$.
\item [\textbf{Instructions}]
\end{description}
\begin{enumerate}
\item If $c \ge a$ then return $L(a,b,c)=1$, else go to step 2.
\item If $a \ | \ b$ then return $L(a,b,c)=\frac{b}{a}$, else go to step 3.
\item Make the positions $a=[a]_{[b]_a}$, $b=[b]_a$ and $c=\dfrac{cb}{c\left \lfloor \frac{b}{a} \right \rfloor+[b]_a}$, then go to step 1.
\item Call $L_1$ the output of step 4.
\item Compute $$L(a,b,c)= \left \lceil \left \lceil  \frac{L_1 (a-c)}{c\left \lfloor \frac{b}{a} \right \rfloor+[b]_a} \right \rceil \frac{b}{a} \right \rceil.$$
\item Return $L(a,b,c)$.

\end{enumerate}
\end{algorithm}

We now show that this algorithm terminates. In fact: 
\begin{proposition}
Algorithm \ref{algoritmo} stops after a finite number of steps.
\end{proposition}
\begin{proof}
Consider the three sequences of integers  $a_i$, $b_i$ and $c_i$ defined recursively as 
\begin{equation*}
b_i=\left\{ 
\begin{array}{ll}
b_0=b & \\
b_i=[b_{i-1}]_{a_{i-1}} & \text{if } i>0
\end{array}
\right.
\end{equation*}
\begin{equation*}
a_i=\left\{ 
\begin{array}{ll}
a_0=a & \\
a_i=[a_{i-1}]_{b_{i}} & \text{if } i>0
\end{array}
\right.
\end{equation*}
\begin{eqnarray*}
c_i=\left\{ 
\begin{array}{ll}
c_0=c & \\
c_i=\dfrac{c_{i-1}b_{i-1}}{c_{i-1}\left \lfloor \frac{b_{i-1}}{a_{i-1}} \right \rfloor+[b_{i-1}]_{a_{i-1}}} & \text{if } i>0
\end{array}
\right.
\end{eqnarray*}

Since it's obvious that $a_{i+1} < a_i$ if $a_i \ge 2$ and that $c_i \ge 1$ for any $i \ge 1$ therefore it follows immediately that after a finite number of steps we will have $a_i \le 1$, hence $c_i \ge a_i$ thus meeting one condition for termination. 

\end{proof}
\section{Applications}
The given algorithm has an application in the context of numerical semigroups. Given two coprime integers $a_1$ and $a_2$ we consider the numerical semigroup generated by them, that is $$S=\langle a_1,a_2 \rangle=\{\lambda_1a_1+\lambda_2a_2 \ | \ \lambda_1,\lambda_2 \in \mathbb{N}\}.$$ We define the \emph{quotient} of a numerical semigroup $S$ by a positive integer $d$ as follows:  $$\frac{S}{d}:=\{x \in \mathbb{N} \ | \ xd \in S \}.$$ The quotient $\frac{S}{d}$ is a numerical semigroup, but it does not have necessarily the same structure as $S$: little is known about the existence of a relation between the invariants of $S$ and $\frac{S}{d}$. In particular, given three positive integers $a_1,a_2,d$, it's an open problem (cf. \cite[Open Problem 5.20]{RG}) to find a formula for the smallest multiple of $d$ that belongs to $\langle a_1,a_2 \rangle$ and for the largest multiple of $d$ that does not belong to $\langle a_1,a_2 \rangle$, that actually are invariants of the quotient $\frac{\langle a_1,a_2 \rangle}{d}$. The class of quotients of numerical semigroups is tightly related to the Diophantine inequalities we have studied, as it was shown that a numerical semigroup is proportionally modular if and only if is the quotient of an embedding dimension two numerical semigroup. In particular we know that $\langle a_1,a_2 \rangle$ is proportionally modular, as the next result shows:  
\begin{lemma}[\protect{\cite[Lemma 18]{RV}}]
Let $a_1,a_2$ be relatively prime positive integers and let $u$ be a positive integer such that $ua_2 \equiv 1 \pmod{a_1}$. Then $$\langle a_1,a_2 \rangle = \{ x \in \mathbb{N} \ | \ [ua_2x]_{(a_1a_2)} \le x \}.$$
\end{lemma}
\noindent By Lemma \ref{gcd} we immediately obtain $$\langle a_1,a_2 \rangle =\left \{ x \in \mathbb{N} \ | \ [ux]_{a_1} \le \frac{x}{a_2} \right \}.$$
Consider now the quotient $$\frac{\langle a_1,a_2 \rangle}{d} = \{x \in \mathbb{N} \ | \ xd \in \langle a_1,a_2 \rangle \}= \left\{ x \in \mathbb{N} \ | \ [uxd]_{a_1} \le \frac{xd}{a_2} \right \}$$ whose multiplicity is
$$	m \left(\frac{\langle a_1,a_2 \rangle}{d}\right) =\min \left\{ x \in \mathbb{N} \ | \ [uxd]_{a_1} \le \frac{xd}{a_2} \right\}=L\left([ud]_{a_1},a_1,\frac{d}{a_2}\right)$$
and therefore it can be obtained by applying Algorithm \ref{algoritmo}.\par The second application regards the set $S(a,b,c)$ itself. Since this set is a numerical semigroup, it has finite complement in $\mathbb{N}$; the greatest integer not belonging to $S(a,b,c)$ is called the \emph{Frobenius number} of $S(a,b,c)$, that we will denote here with $F(a,b,c)$. In \cite{RV2} the authors give a relation between $F(a,b,1)$ and the multiplicity of a particular proportionally modular numerical semigroup. For this purpose we fix the following notation.\\  Given $p,q \in \mathbb{Q}^+$ such that $p<q$ denote by $[p,q]$ and $\langle [p,q] \rangle$ the sets $$[p,q]=\{x \in \mathbb{Q} \ | \ p \le x \le q\}$$
and $$\langle[p,q]\rangle=\{\lambda_1a_1 + \lambda_2a_2+\ldots+\lambda_na_n | \ \lambda_1,\ldots,\lambda_n \in \mathbb{N}, \ \ a_1,\ldots,a_n \in [p,q] \ \ n \in \mathbb{N}\setminus \{0\}\}$$ It is known that for any $p,q \in \mathbb{Q}^+$ such that $p<q$ the set $S([p,q])=\langle [p,q] \rangle \cap \mathbb{N}$ is a proportionally modular numerical  semigroup, as the next proposition shows:
\begin{proposition}[\protect{\cite[Proposition 1]{RV2}}]\label{ratio}
Let $a_1,b_1,a_2,b_2 \in \mathbb{Z}^+$ be such that $\frac{b_1}{a_1} < \frac{b_2}{a_2}$. Then $S([\frac{b_1}{a_1},\frac{b_2}{a_2}])=S(a_1b_2,b_1b_2,a_1b_2-a_2b_1).$
\end{proposition}
A direct consequence of Proposition \ref{ratio} is that $m(S([\frac{b_1}{a_1},\frac{b_2}{a_2}]))=L(a_1b_2,b_1b_2,a_1b_2-a_2b_1)$. Furthermore, note that Lemma \ref{gcd} allows us to divide each term by $b_2$, hence obtaining
\begin{equation}\label{m}
m\left(S\left(\left[\frac{b_1}{a_1},\frac{b_2}{a_2}\right]\right)\right)=L\left(a_1,b_1,\frac{a_1b_2-a_2b_1}{b_1}\right)
\end{equation}   
\begin{theorem}[\protect{\cite[Theorem 18]{RV2}}]\label{frob}
Let $a,b \in \mathbb{Z}^+$ be such that $2 \le a < b$ and $S=S([\frac{2b^2+1}{2ab},\frac{2b^2-1}{2b(a-1)}])$. Then $F(a,b,1)=b-m(S)$.
\end{theorem}
By Theorem \ref{frob} and Eq. (\ref{m}) we have $$F(a,b,1)=b-m(S)=b-L\left(2b,2b^2+1,\frac{4b^3-4ab+2b}{2b^2+1}\right)$$ and thus we can apply Algorithm \ref{algoritmo}.

\end{document}